\documentclass[12pt,reqno]{amsart}
\usepackage{graphicx, xcolor, enumerate}
\usepackage{amsmath, amssymb}
\usepackage{here}
\usepackage{pstool}
\usepackage{amsthm} 
\usepackage{newtxmath,newtxtext}
\usepackage{bm}
\usepackage{lineno}
\usepackage{bigints}
\usepackage{caption}
\usepackage{color}
\usepackage{graphicx,color}
\usepackage{amsmath, amssymb, graphics}

\theoremstyle{plain} 
\newtheorem{theorem}{Theorem}[section]
\newtheorem{lemma}[theorem]{Lemma}
\newtheorem{corollary}[theorem]{Corollary}
\newtheorem{proposition}[theorem]{Proposition}
\newtheorem{fact}[theorem]{Fact}

\theoremstyle{definition}
\newtheorem{definition}[theorem]{Definition}
\newtheorem{remark}[theorem]{Remark}

\newtheorem*{maintheorem*}{Main Theorem}

\numberwithin{equation}{section}

\newenvironment{contentsinred}{\bgroup\color{magenta}}{\egroup}
\newenvironment{contentsinred2}{\bgroup\color{brown}}{\egroup}

\newcommand{\by}{\begin{contentsinred}}  
\newcommand{\ey}{\end{contentsinred}}

\newcommand{\hthree}{\mathbb{H}^3}

\newcommand{\sltc}{\mathrm{SL}(2,\mathbb{C})}
\newcommand{\hermtwo}{\mathrm{Herm}(2,\mathbb{C})}

\newcommand{\qthreep}{\mathbb{Q}^3_+}
\newcommand{\qthreem}{\mathbb{Q}^3_-}
\newcommand{\qthreept}{\mathring{\mathbb{Q}}^3_+}
\newcommand{\lfour}{\mathbb{L}^4}
\newcommand{\ithree}{\mathbb{I}^3}

\newcommand{\divz}{\operatorname{d}\!z}

\newcommand{\G}{\mathrm{G}}

\newcommand{\tr}[1]{\operatorname{tr} #1}

\newcommand{\h}{H}
\newcommand{\mcU}{\mathcal{U}}

\begin{document}
\title[Bernstein-type Theorems for CMC surfaces in the light cone]{Bernstein-type Theorems \\ for constant mean curvature surfaces \\ in the three-dimensional light cone}
\author[S. Akamine]{Shintaro Akamine}
\address[Shintaro Akamine]{
	College of Bioresource Sciences,
	Nihon University,
	1866 Kameino, Fujisawa, Kanagawa, 252-0880, Japan}
\email{akamine.shintaro@nihon-u.ac.jp}

\author[W. Lee]{Wonjoo Lee}
\address[Wonjoo Lee]{
	Department of Mathematics, Jeonbuk National University, Jeonju-si, Jeonbuk State 54896, Republic of Korea}
\email{w\_lee@jbnu.ac.kr}

\author[S-D. Yang]{Seong-Deog Yang}
\address[Seong-Deog Yang]{
	Department of Mathematics, Korea University, Seoul 02841, Republic of Korea}
\email{sdyang@korea.ac.kr}

\begin{abstract}
	We establish Bernstein-type theorems for entire constant mean curvature graphs in the three-dimensional light cone $\qthreep$ over the horosphere under the assumption that the Gaussian curvature $K$ is bounded below, by showing that such graphs are horospheres or spheres of $\qthreep$.
\end{abstract}

\maketitle
\section{Introduction}\label{Sec.1}
Bernstein's Theorem is one of the most beautiful theorems in the theory of minimal surfaces. Since the inception of the Theorem in the early 1920's, there have been numerous generalizations of it in various contexts.
For example, do Carmo and Lawson showed in \cite{doCL} that any non-parametric hypersurface of constant mean curvature in hyperbolic $n$-space $\mathbb{H}^n$ which is defined over an entire totally geodesic hyperplane $\mathbb{H}^{n-1}$ is a hypersphere. Here, a hypersphere is one of the totally umbilic hypersurfaces in $\mathbb{H}^n$, and it appears as a hypersurface of constant distance from a totally geodesic hypersurface. A hypersphere also arises, in the Minkowski space model of $\mathbb{H}^n$, as the intersection of $\mathbb{H}^n$ with a timelike hyperplane in Minkowski space $\mathbb{L}^{n+1}$ (see \cite{Spivak} for example).

On the other hand, totally umbilic hypersurfaces in $\mathbb{H}^n$ include not only hyperspheres but also hypersurfaces called spheres and horospheres, which arise as the intersections of $\mathbb{H}^n$ with spacelike or lightlike hyperplanes in $\mathbb{L}^{n+1}$, respectively.
In the upper half space model of $\mathbb{H}^n$, which is realized as
\[
	\mathbb{R}^n_+=\left \{(x_1,x_2, \cdots, x_n)\in \mathbb{R}^{n-1}\times \mathbb{R}^+\colon \frac{dx_1^2+dx_2^2+ \cdots + dx_{n}^2}{x_n^2}\right\},
\]
horospheres can be written as $x_{n+1}=t_0$ for some constant $t_0>0$ under a suitable isometry. Hence, Do Carmo--Lawson \cite{doCL} and Koh \cite{Koh} also considered the Bernstein-type problem for graphs over a horosphere of the form $x_{n+1}=f(x_1,x_2,\cdots,x_n)$, and proved that any entire constant mean curvature (CMC) graph defined over the whole $\mathbb{R}^n$ must be a horosphere, that is, $f$ must be a constant function. Furthermore, it is shown in \cite{Koh} that if the orthogonal projection from a non-zero CMC hypersurface into a horosphere is not surjective and its image is simply connected, then hypersurface is a hypersphere.

In this article, we investigate Bernstein-type Theorems for CMC surfaces in the three-dimentional light cone $\qthreep$.
Interestingly, $\qthreep$ minus a lightlike line can be represented as
$$
	\left\{ (u,v,t) \in \mathbb{R}\times\mathbb{R}\times \mathbb{R}^+ : ds^2= \frac{du^2+dv^2}{t^2}\right\},
$$
which we call the upper half space model of $\qthreep$ \cite{Cho}.
As in $\hthree$ \cite{doCL}, a sphere, a horosphere, and a hypersphere in $\qthreep$ are defined as the intersection of $\qthreep$ with a spacelike hyperplane, a lightlike hyperplane, and a timelike hyperplane in $\lfour$, respectively.
See Definition~\ref{Def:202601140823PM} for details.
In the upper half space model, the horospheres of $\qthreep$ are again represented, up to isometry, in the form $t=t_0$ for some constant $t_0>0$. However, unlike the case of hyperbolic space, the horospheres in $\qthreep$ are not only totally umbilic but are also characterized as totally geodesic surfaces, as shown in Proposition \ref{Prop:202601150810PM}. Therefore, we may consider graphs over such a horosphere as a Bernstein-type problem in $\qthreep$.

Bernstein-type Theorems in a degenerate metric space such as $\qthreep$ may require additional assumptions other than being  entire and minimal. For example, as was pointed out in \cite{ALY1}, there are many entire zero mean curvature (ZMC) surfaces in the isotropic three space $\ithree$. An appropriate condition is a lower bound for the Gaussian curvature $K$, and we obtain the following:

\begin{maintheorem*}\label{thm:Berntypeq3+}
	\begin{itemize}
		\item[(1)] Any entire ZMC graph in $\mathbb{Q}^3_+$ has Gaussian curvature bounded below
		      if and only if it is the image of $t=t_0$ for some positive constant $t_0$.
		      It is a horosphere of $\qthreep$.

		\item[(2)] Any entire CMC $ \h <0$  graph in $\mathbb{Q}^3_+$ has Gaussian curvature bounded below
		      if and only if it is congruent to the image of
		      $
			      t = \sqrt{\frac{-\h}{2} } (u^2 + v^2 + 1).
		      $
		      It is a punctured sphere of $\qthreep$.
		\item[(3)]  In $\mathbb{Q}^3_+$, there exists no entire graph with positive CMC.
	\end{itemize}
\end{maintheorem*}
Interestingly,  unlike the case of $\mathbb{H}^3$, a hypersphere in $\qthreep$ does not provide a solution to the Bernstein problem. \cite{doCL} and \cite{Koh} do not assume curvature bounds, but we cannot relax the condition that $K$ is bounded below. See Remark~\ref{Rem:202601180320PM}.

We can also obtain a characterization of spheres as graphs over the entire ideal boundary of $\qthreep$. See Figure~\ref{Fig:20260937PM}.
Given $p \in \qthreep \subset \lfour$, let $[p]$ be the lightlike ray emanating from the origin of $\lfour$ and going through $p$.
The ideal boundary $\partial^\infty\qthreep$ of $\qthreep$ can be identified with the set of all such rays.
Let $\pi$ denote the natural projection
$$
	\pi : \qthreep \to \partial^\infty\qthreep, \qquad
	p \mapsto [p].
$$
Using $\pi$, we may consider graphs over a domain in the ideal boundary.
In particular, we consider graphs over a punctured ideal boundary, by which we mean $\partial^\infty\qthreep\setminus\{[p_0]\}$ for some $p_0 \in \qthreep$.

Any punctured ideal boundary with the puncture $[p]$ can be identified via $\pi$ with a family of horospheres touching the ideal boundary at $[p]$. So the results of the Main Theorem can be rephrased as a Bernstein-type theorem for graphs over a punctured ideal boundary of $\qthreep$.

\begin{corollary}\label{Coro:202603071137AM}
	Let $S$ be a spacelike CMC $H$ surface in $\qthreep$ such that $\pi|_S$ is a one-to-one correspondence between $S$ and the entire ideal boundary. Then, $H<0$ and $S$ is a sphere.
\end{corollary}

\begin{figure}[t]
\begin{center}
\psfragfig[width=0.35\textwidth]{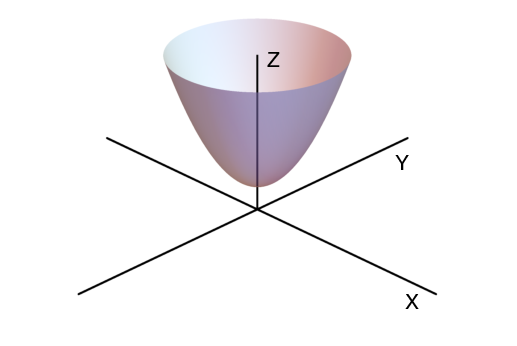}{
    \psfrag{X}{$u$}
    \psfrag{Y}{$v$}
    \psfrag{Z}{$t$}
    \psfrag{U}{$\tilde{u}$}
    \psfrag{V}{$\tilde{v}$}
    \psfrag{W}{$\tilde{w}$}
  }%
  \hspace{1cm}
\psfragfig[width=0.35\textwidth]{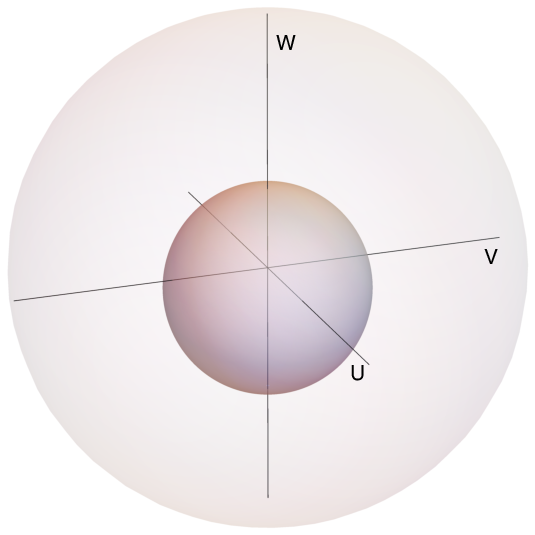}{
    \psfrag{X}{$u$}
    \psfrag{Y}{$v$}
    \psfrag{Z}{$t$}
    \psfrag{U}{$\tilde{u}$}
    \psfrag{V}{$\tilde{v}$}
    \psfrag{W}{$\tilde{w}$}
  }
	\caption{ A sphere in the half space model (left)
and in the punctured ball model (right) of $\qthreep$.}
\label{Fig:20260937PM}
\end{center}
\end{figure}

\section*{Acknowledgements}

We would like to thank Joseph Cho for sharing his insights. The first author was supported by JSPS KAKENHI Grant Number 23K12979. The second and third authors were supported by the NRF of Korea funded by MSIT (Korea-Austria Scientific and Technological Cooperation RS-2025-1435299, P.I.: Joseph Cho).

\section{Preliminaries}\label{Sec.2}

In this section, we will briefly review the basic geometry of the three-dimensional light cone $\qthreep$ and its surface
theory.

\subsection{Hermitian model of $\qthreep$}

Let $\lfour$ denote the four-dimensional Lorentzian space,  whose inner product is
\[
	\langle (x_0, x_1, x_2, x_3), (y_0, y_1, y_2, y_3) \rangle := - x_0 y_0 + x_1 y_1 + x_2 y_2 + x_3 y_3.
\]
We identify $\mathbb{L}^4$ with $\hermtwo$ via the map :
\[
	\lfour \ni (x_0, x_1, x_2, x_3) \sim \begin{pmatrix} x_0 + x_3 & x_1 + i x_2 \\ x_1 - i x_2 & x_0 - x_3 \end{pmatrix} = X \in \hermtwo.
\]
Then the inner product of $X, Y$ can be written as follows:
\[
	\langle X, Y \rangle = -\frac{1}{2} \left(\det{(X+Y)} - \det{X} - \det{Y} \right).
\]
The three-dimensional light cones are
\begin{align*}
	\qthreep & := \{ X \in \hermtwo | \langle X, X \rangle = 0, \tr{X} > 0 \}, \\
	\qthreem & := \{ X \in \hermtwo | \langle X, X \rangle = 0, \tr{X} < 0 \}.
\end{align*}

For any $A \in \sltc$, the following map
$$
	\varphi_A: \hermtwo \ni X \mapsto A X A^* \in \hermtwo,
$$
where $A^*$ is the conjugate transpose of $A$, is an orientation preserving isometry of $\lfour$ which
preserves $\qthreep$.
We let $\varphi_-(x_0,x_1,x_2,x_3) := (x_0,x_1,x_2,-x_3)$, and let
$$
	\operatorname{Isom}(\qthreep) := \operatorname{span} \langle (\varphi_A)|_{\qthreep}, \varphi_- \rangle_{A \in \sltc}
$$
be the set of all isometries of $\qthreep$, where $(\varphi_A)|_{\qthreep}$ is the restriction of $\varphi_A$ to $\qthreep$.
We say that two objects $S_1$, $S_2$ in $\qthreep$ are congruent to each other if $S_2= \varphi(S_1)$ for some $\varphi \in \operatorname{Isom}(\qthreep)$.

\subsection{Punctured ball model of $\qthreep$}
Using the stereographic projection
$$
	\Pi : \qthreep \to \mathbb{R}^3,  \quad
	\begin{pmatrix} x_0 + x_3 & x_1 + i x_2 \\ x_1 - i x_2 & x_0 - x_3 \end{pmatrix} \mapsto
	(\tilde{u},\tilde{v},\tilde{w}) = \left( \frac{x_1}{1+x_0}, \frac{x_2}{1+x_0}, \frac{x_3}{1+x_0} \right)
$$
from $(x_0,x_1,x_2,x_3)=(-1,0,0,0)$, we can ideitify $\qthreep$, $\partial^\infty\qthreep$, and the vertex of $\qthreep$ with
\begin{align*}
	B & := \{ (\tilde{u},\tilde{v},\tilde{w}) \in \mathbb{R}^3 : 0 < \tilde{u}^2+\tilde{v}^2+\tilde{w}^2< 1\}, \\
	S & := \{ (\tilde{u},\tilde{v},\tilde{w}) \in \mathbb{R}^3 : \tilde{u}^2+\tilde{v}^2+\tilde{w}^2 = 1\},
\end{align*}
and the origin $(0,0,0)$, respectively.
For the ideal boundary $\partial^\infty\qthreep$, see the Introduction.
We call $B$ a {\it punctured ball model} of $\qthreep$ (\textit{cf}. \cite{CLLY}).

\subsection{Half space model of $\qthreep$ and graphs over a horosphere}

For convenience, we introduce
$$
	\qthreept := \qthreep \setminus \{ (x_0,x_1,x_2,x_3) : x_0-x_3 = x_1 = x_2 = 0\}.
$$

Then
\begin{equation}\label{Eq:202601170935PM}
	\Phi: \mathbb{R}\times \mathbb{R}\times\mathbb{R}^+ \ni (u,v,t) \mapsto \frac{1}{t} \begin{pmatrix} u^2+v^2 & u+iv \\ u-iv & 1 \end{pmatrix} \in \qthreept
\end{equation}
is a one-to-one correspondence and the pull back of the metric by $\Phi$ is
\begin{equation}\label{202602210753AM}
	ds^2 = \frac{du^2+dv^2}{t^2}.
\end{equation}

By abusing terminology, we call $\mathbb{R}^2 \times \mathbb{R}^+$ equipped with this metric \textit{the half space model} of $\qthreep$. To the authors' knowledge, this model was first noticed by Joseph Cho \cite{Cho}.

Note that
\begin{gather*}
	\Pi \circ \Phi[\mathbb{R}^2\times\mathbb{R}^+] = B\setminus\{(0,0,z):  0 < z < 1 \} \cong \qthreept, \\
	\Pi \circ \Phi[\mathbb{R}^2\times\{0\}] =S\setminus\{(0,0,1)\} \cong \partial^\infty\qthreep\setminus\{[(1,0,0,1)]\}.
\end{gather*}

\begin{figure}[t]
\begin{center}
\psfragfig[width=0.9\linewidth]{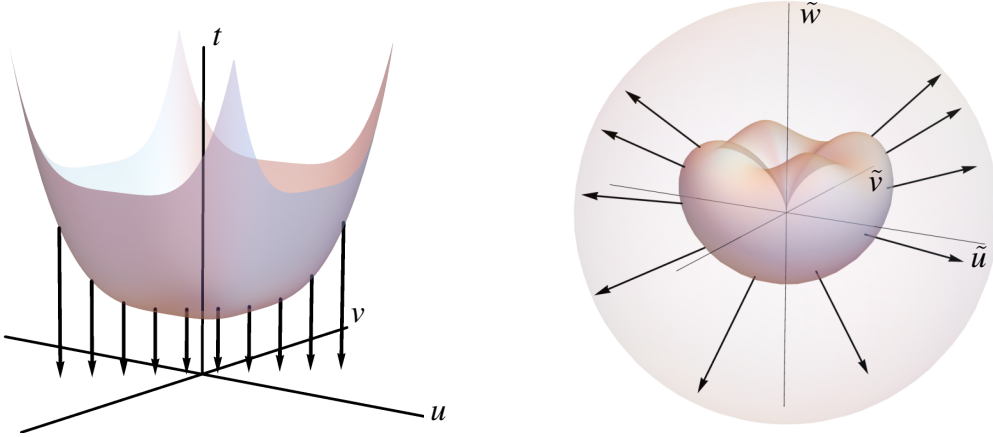}{
    \psfrag{X}{$u$}
    \psfrag{Y}{$v$}
    \psfrag{Z}{$t$}
    \psfrag{U}{$\tilde{u}$}
    \psfrag{V}{$\tilde{v}$}
    \psfrag{W}{$\tilde{w}$}
  }
\caption{The projection $\pi$ and a graph in the half space  model (left) and in the punctured ball model (right).}
\label{Fig:202603161139}
\end{center}
\end{figure}

We define the notion of graphs as follows. See also Figure~\ref{Fig:202603161139}.
\begin{definition}
	Given a function $t = \tau(u,v)$ on $\mathcal{U} \subset \mathbb{R}^2$, we call the image of
	\begin{equation}\label{Eq:202505110555AM}
		X(u,v) := \Phi(u,v, \tau(u,v))
		= \frac{1}{\tau(u,v)}\begin{pmatrix} u^2+v^2 & u+iv \\ u-iv & 1 \end{pmatrix},
		\qquad (u,v) \in \mathcal{U}
	\end{equation}
	the \textit{graph} of $\tau$ over $\mathcal{V} := \{ \Phi(u,v,1): (u,v) \in \mathcal{U}\}$ (which is a part of the standard horosphere. See Definition~\ref{Def:202601140823PM}) in $\qthreep$. If $\mathcal{U}= \mathbb{R}^2$ we call it \textit{entire}.
\end{definition}

Strictly speaking, the surface given by \eqref{Eq:202505110555AM} is a graph in $\qthreept$. However, we remark that any conformal spacelike immersion into $\qthreep$ is locally written as in \eqref{Eq:202505110555AM} after an isometry of $\qthreep$, see \cite{Liu2} for example.

\subsection{Curvatures of graphs in $\qthreep$}
For computational purposes, we use the Hermitian model of $\qthreep$.

Given a spacelike immersion $X: \mcU \subset \mathbb{R}^2 \to \qthreep$ with coordinates $(u,v) = (u^1, u^2) \in \mathcal{U}$, there is a unique map $\G : \mcU \to \qthreem$ which satisfies
\begin{align}\label{Eq:202507142137}
	\langle \G, \G \rangle = \langle \G, X_u \rangle = \langle \G, X_v \rangle = 0,\qquad \langle \G, X \rangle = 1,
\end{align}
called the \emph{lightlike Gauss map} of $X$ (\textit{cf}. \cite{IPR2004}, \cite{Liu2}).
The first and  the  second fundamental forms of $X$ are given by
\[
	\mathbf{g} :=   g_{ij} \operatorname{d}\!u^i \operatorname{d}\!u^j = \langle X_i, X_j \rangle \operatorname{d}\!u^i \operatorname{d}\!u^j , \quad
	\mathbf{A} :=  A_{ij} \operatorname{d}\!u^i \operatorname{d}\!u^j =\langle \G, X_{ij} \rangle \operatorname{d}\!u^i \operatorname{d}\!u^j,
\]
respectively.

Given a spacelike immersion $X$, suppose that $u,v$ are conformal parameters, i.e.
\begin{equation}\label{eq:fff}	\mathbf{g}=e^{2{\omega}}(\operatorname{d}\!u^2+\operatorname{d}\!v^2)=e^{2{\omega}}\operatorname{d}					\!z\operatorname{d}\!\bar{z}         \qquad        z:=u+iv
\end{equation}
for some function $\omega \colon \mathcal{U} \subset \mathbb{R}^2  \to \mathbb{R}$. Then
\[
	H =2e^{-2 \omega } \langle \G, X_{z\bar{z}} \rangle, \qquad
	Q\operatorname{d}\!z^2 = \langle \G, X_{zz} \rangle\divz^2, \qquad
	K=H^2-4Q\bar{Q}e^{-4{\omega}}
\]
for the mean curvature, the Hopf differential and the Gaussian curvature of $X$, respectively.
The Gauss-Weingarten equations (\textit{cf}. \cite{Liu1}) are
\begin{equation}\label{Eq:202601281208PM}
	\begin{cases}
		X_{zz} = Q X + 2 \omega_z X_z,                                      \\
		X_{z\bar{z}} = \frac{1}{2}He^{2\omega}X - \frac{1}{2}e^{2\omega}\G, \\
		\G_z = -HX_z - 2Qe^{-2 \omega}X_{\bar{z}}
	\end{cases}
\end{equation}
and the Gauss-Codazzi equations are
\[\begin{cases}
		2\omega_{z\bar{z}}=He^{2 \omega}, & \text{(Gauss equation)}   \\
		H_z = 2 Q_{\bar{z}} e^{-2\omega}. & \text{(Codazzi equation)}
	\end{cases}\]

If  $H$ is constant, we call the surface a \emph{constant mean curvature (CMC) surface}.
If  $H \equiv 0$, we call it a \emph{zero mean curvature (ZMC) surface}.
So a ZMC surface is a CMC surface.

By the Gauss equation, $H$ is an intrinsic invariant and the Codazzi equation implies that $H$ is a constant if and only if $Q$ is holomorphic. In particular, if $X$ has ZMC, then the metric $\mathbf{g}$ is a flat metric.

Now we consider the graph of $t=\tau(u,v)$, i.e. the image of \eqref{Eq:202505110555AM}.
Direct computations show that
\begin{gather}
	\label{Eq:202601180922AM}
	(\mathbf{g}_{ij}) = \begin{pmatrix} 1/\tau^2 & 0 \\ 0 & 1/\tau^2 \end{pmatrix}, \qquad
	(\mathbf{A}_{ij}) = \frac{\tau_u^2+\tau_v^2}{2\tau^2 } \begin{pmatrix} 1 & 0 \\ 0 & 1 \end{pmatrix}
	-\frac{1}{\tau} \begin{pmatrix} \tau_{uu} & \tau_{uv} \\ \tau_{vu} & \tau_{vv} \end{pmatrix},
	\\
	\label{Eq:202601180923AM}
	Q = - \frac{\tau_{zz} }{\tau}, \qquad
	H = - 2 \tau^2 (\ln \tau)_{z\bar{z}}, \qquad
	H^2-K = 4 \tau^2 |\tau_{zz}|^2 \ge 0.
\end{gather}

\section{Spheres, horospheres, and hyperspheres in $\qthreep$}\label{Spheres}

\subsection{Intersection of $\qthreep$ with hyperplanes in $\mathbb{L}^4$}\label{SubSec:202601151013PM}
Recall that  spheres, horospheres, and hyperspheres in the hyperbolic three-space $\hthree$ as the hyperboloid in $\lfour$  are obtained as the intersection of $\hthree$ with spacelike hyperplanes, lightlike hyperbolic planes, and timelike hyperplanes, respectively.

In this  section, we consider the intersection of $\qthreep$ with hyperplanes of $\lfour$ and investigate their properties.

\begin{definition}\label{Def:202601140823PM}
	We call the intersection of a hyperplane of $\lfour$ with $\qthreep$ a \textit{sphere}, a \textit{horosphere}, or a \textit{hypersphere}
	if the hyperplane is spacelike, lightlike, or timelike, respectively, and the intersection is a regular surface.
	We call the intersection the standard horosphere if the hyperplane satisfies $x_0-x_3=1$.
\end{definition}

Direct calculations show that all of them are totally umbilic. In the next Proposition, we show that the converse is also true, following a standard argument (\textit{cf}. \cite[Ch~7, Theorem~29]{Spivak} for example).
\begin{proposition}\label{Prop:202601150810PM}
	Any connected totally umbilic surface in $\qthreep$ is a (part of a) sphere, horosphere, or a hypersphere.
	In particular, any connected totally geodesic surface in $\qthreep$ is a (part of a) horosphere.
\end{proposition}\label{Prop:202601150642AM}

\begin{proof}
	Suppose $X$ is totally umbilic. Then, it follows from \eqref{Eq:202601281208PM} that for any $(u,v)$
	$$
		G_u(u,v) = -H(u,v) X_u(u,v), \qquad
		G_v(u,v) = -H(u,v) X_v(u,v)
	$$
	for the mean curvature function $H(u,v) $.
	From $G_{uv} = G_{vu}$, we see that $H_u=H_v=0$ for all $(u,v)$, so $H$ is constant, say $\h$.

	Suppose that $\h = 0$. Then $G$ is constant, say, $G(u,v)=G_0$. Then, from \eqref{Eq:202507142137}, we see that $\langle X, G_0 \rangle = 1$,
	hence $X$ lies in a lightlike hyperplane.

	Suppose that $\h \not= 0$. Then
	$
		G(u,v) = -\h ( X(u,v) - \vec{d} )
	$
	for some constant vector $\vec{d}$.
	Hence
	$$
		0 = \langle G, G \rangle = \h^2 \langle X-\vec{d}, X-\vec{d} \rangle = \h^2 ( \langle X, X \rangle - 2 \langle X, \vec{d} \rangle + \langle \vec{d},\vec{d} \rangle).
	$$
	Therefore,
	$
		\langle X, \vec{d} \rangle = \tfrac{1}{2} \langle \vec{d}, \vec{d} \rangle.
	$
	So $X$ lies in a hyperplane.
	Note that $\vec{d} = X + \tfrac{1}{\h} G$ hence
	$
		\langle \vec{d}, \vec{d} \rangle = 2/\h.
	$
	Therefore, if $\h<0$, then $\vec{d}$ is timelike and the hyperplane is spacelike.
	If $\h>0$, then $\vec{d}$ is spacelike and the hyperplane is timelike. Now the conclusion follows.
\end{proof}

For an arbitrary constant $t>0$, consider the hyperplane
\begin{equation}\label{Eq:202601260727PM}
	\Pi_t : x_0 - x_3 = 1/t.
\end{equation}
The following map
$$
	\begin{pmatrix} \sqrt{t} & 0 \\ 0 & 1/\sqrt{t} \end{pmatrix}
	\begin{pmatrix} u^2 + v^2 & u + i v \\ u - i v & 1 \end{pmatrix}
	\begin{pmatrix} \sqrt{t} & 0 \\ 0 & 1/\sqrt{t} \end{pmatrix}^*
	= \frac{1}{t} \begin{pmatrix} \tilde{u}^2 + \tilde{v}^2 & \tilde{u} + i \tilde{v} \\ \tilde{u}  - i \tilde{v}  & 1 \end{pmatrix}
$$
where $\tilde{u} := t u$, $\tilde{v} := t v$ shows that $\Pi_t$ is isometric to $\Pi_1$ by an origin fixing isometry of $\mathbb{L}^4$.
Using this observation, we can easily conclude that all horospheres in $\qthreep$ are congruent to each other.

\begin{remark}
	Note that any horosphere in $\mathbb{H}^3$ is also obtained as the intersection of $\mathbb{H}^3 \subset \mathbb{L}^4$ and a lightlike plane.
	However, horospheres in $\mathbb{H}^3$ are totally umbilic but not totally geodesic, while horospheres in $\qthreep$ are totally geodesic.
\end{remark}

\subsection{Graph functions of the hyperplane sections}
An arbitrary hyperplane $\Pi$ in $\mathbb{L}^4$ has an equation of the following form:
\begin{align}\label{Eq:202506180738}
	a (x_0+x_3) + b(x_0-x_3) + cx_1 + d x_2 =  k
\end{align}
for $a,b,c,d,k\in\mathbb{R}$. Consider $\Pi \cap \qthreep$. If $k=0$ in \eqref{Eq:202506180738} then $\Pi$ passes through the origin, and $\Pi \cap \mathbb{Q}^3_+$ is not spacelike. Since we are interested in spacelike surfaces, we  may  assume without loss of generality that $k = 1$ such that
the equation of $\Pi$ is
\begin{align}\label{Eq:202507191432}
	a (x_0+x_3) + b(x_0-x_3) + cx_1 + d x_2 =  1.
\end{align}
Then  $\vec{n} := (-(a+b), c, d , a-b)^t$ is normal to  the hyperplane $\Pi$ given by \eqref{Eq:202507191432}, and satisfies
\[
	\langle \vec{n}, \vec{n}  \rangle  = - 4 ab + c^2 + d^2.
\]
So $\Pi$ is spacelike, lightlike, or timelike if and only if $- 4 ab + c^2 + d^2 <, =$, or $>0$, respectively.

In visualizing $\Pi \cap \qthreept$, it is better to consider the half space model of $\qthreept$.
Since an arbitrary point of $\qthreept$ can be written as follows
\begin{align}\label{202506180754PM}
	\begin{pmatrix} x_0+x_3 & x_1+ix_2 \\ x_1-i x_2 & x_0-x_3 \end{pmatrix}
	= \frac{1}{t} \begin{pmatrix} u^2 + v^2 & u + i v \\ u - i v & 1 \end{pmatrix},
\end{align}
an  arbitrary point $(u,v,t)$ of  $\Phi^{-1} (\Pi \cap \qthreept)$  with $\Phi$ in \eqref{Eq:202601170935PM} satisfies

\begin{align}\label{Eq:202506190732AM}
	t = a(u^2+v^2)+ b + cu + dv \overset{(*)~}{=} a(u+\tfrac{c}{2a})^2 + a (v + \tfrac{d}{2a})^2 + \tfrac{4ab-c^2-d^2}{4a}.
\end{align}
The second equality $(*)$ assumes $a\not=0$.
So $\Pi \cap \qthreept$ is a graph over a subset of $\Pi_1$ (\textit{cf}. \eqref{Eq:202601260727PM}).

By analyzing the sign of $4ab-c^2-d^2$ and the positivity of $t$ in \eqref{Eq:202506190732AM}, we obtain the following assertions.
\begin{lemma}\label{Lem:202601171034AM}
	The graph of \eqref{202506180754PM} with \eqref{Eq:202506190732AM} is (a part of) \\
	\  (S) a sphere  iff  (S-i)  $a>0$, $b>\tfrac{c^2+d^2}{4a}$, \\
	\  (L) a horosphere iff  (L-i) $a=0$, $b>0$, $c=d=0$ or \ (L-ii) $a > 0$, $b= \frac{c^2+d^2}{4a}$. \\
	\  (T) a hypersphere iff (T-i) $a=0$, $c^2+d^2>0$, (T-ii) $a >0$, $b<\tfrac{c^2+d^2}{4a}$,
	or  (T-iii) $a<0$, $b>\tfrac{c^2+d^2}{4a}$. \\
	In all other cases for $a,b,c,d$, we have that $\Pi \cap \qthreept = \emptyset.$
\end{lemma}
See Figure~\ref{Fig:202603160953PM}.
\begin{lemma}\label{Lem:202601171033AM}
	The graph of \eqref{Eq:202506190732AM} is entire if and only if (S-i)  $a>0$, $b>\tfrac{c^2+d^2}{4a}$ or (L-i) $a=0$, $b>0$, $c=d=0$.
\end{lemma}

\begin{figure}[t]
\begin{center}
\psfragfig[width=0.25\textwidth]{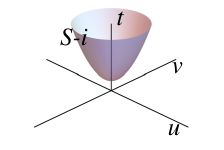}{
    \psfrag{X}{$u$}
    \psfrag{Y}{$v$}
    \psfrag{Z}{$t$}
    \psfrag{U}{$\tilde{u}$}
    \psfrag{V}{$\tilde{v}$}
    \psfrag{W}{$\tilde{w}$}
    \psfrag{Si}{\textit{S-i}}
  }%
  \hfill
\psfragfig[width=0.25\textwidth]{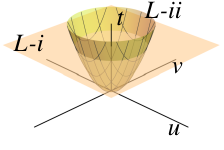}{
    \psfrag{X}{$u$}
    \psfrag{Y}{$v$}
    \psfrag{Z}{$t$}
    \psfrag{U}{$\tilde{u}$}
    \psfrag{V}{$\tilde{v}$}
    \psfrag{W}{$\tilde{w}$}
    \psfrag{Li}{\textit{L-i}}
    \psfrag{Lii}{\textit{L-ii}}
  }%
  \hfill
\psfragfig[width=0.25\textwidth]{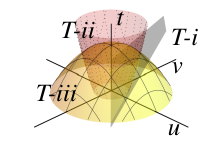}{
    \psfrag{X}{$u$}
    \psfrag{Y}{$v$}
    \psfrag{Z}{$t$}
    \psfrag{U}{$\tilde{u}$}
    \psfrag{V}{$\tilde{v}$}
    \psfrag{W}{$\tilde{w}$}
    \psfrag{Ti}{\textit{T-i}}
    \psfrag{Tj}{\textit{T-ii}}
    \psfrag{Tk}{\textit{T-iii}}
  }
  \bigskip
\psfragfig[width=0.25\textwidth]{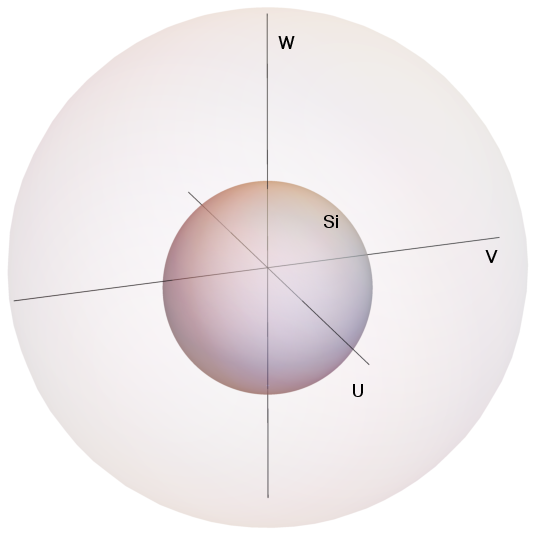}{
    \psfrag{X}{$u$}
    \psfrag{Y}{$v$}
    \psfrag{Z}{$t$}
    \psfrag{U}{$\tilde{u}$}
    \psfrag{V}{$\tilde{v}$}
    \psfrag{W}{$\tilde{w}$}
    \psfrag{Si}{\textit{S-i}}
  }%
  \hfill
\psfragfig[width=0.25\textwidth]{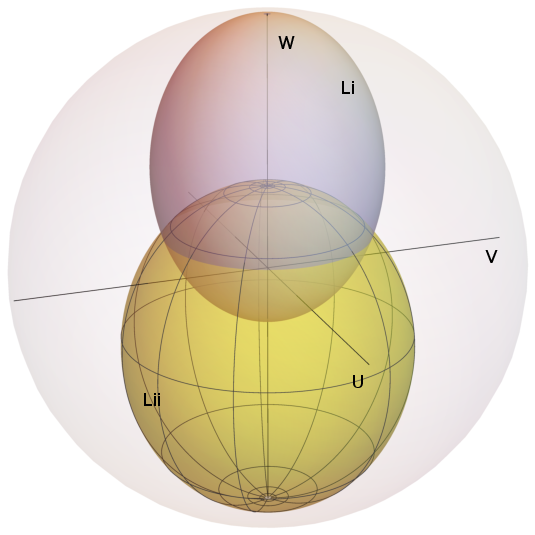}{
    \psfrag{X}{$u$}
    \psfrag{Y}{$v$}
    \psfrag{Z}{$t$}
    \psfrag{U}{$\tilde{u}$}
    \psfrag{V}{$\tilde{v}$}
    \psfrag{W}{$\tilde{w}$}
    \psfrag{Li}{\textit{L-i}}
    \psfrag{Lii}{\textit{L-ii}}
  }%
  \hfill
\psfragfig[width=0.25\textwidth]{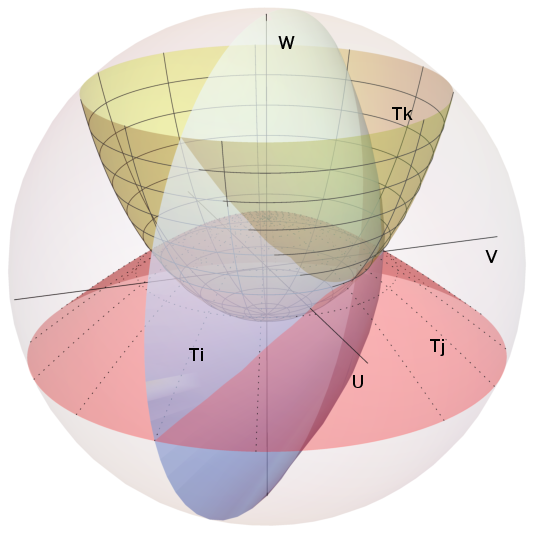}{
    \psfrag{X}{$u$}
    \psfrag{Y}{$v$}
    \psfrag{Z}{$t$}
    \psfrag{U}{$\tilde{u}$}
    \psfrag{V}{$\tilde{v}$}
    \psfrag{W}{$\tilde{w}$}
    \psfrag{Ti}{\textit{T-i}}
    \psfrag{Tj}{\textit{T-ii}}
    \psfrag{Tk}{\textit{T-iii}}
  }
  \caption{
    A sphere, two horospheres, and three hyperspheres in $\qthreep$, in the half space model (top row) and in the punctured ball model (bottom row). For labels such as {\it S-i} in the figure, see Lemma~\ref{Lem:202601171034AM}.
  }
\label{Fig:202603160953PM} 
\end{center}
\end{figure}

\textit{ (S-i) } is equivalent to saying that $a>0$ and that the plane $\Pi$ is spacelike.
\textit{(L-i) } is equivalent to saying that $\Pi$ is a lightlike hyperplane parallel to $\Pi_1$.

\textit{(L-i)} and \textit{(L-ii)} show that the horospheres touch the ideal boundary tangentially only once.
Note that if
\begin{equation}\label{Eq:202603060429PM}
	X(u,v) := \left( u,v, a(u^2+v^2) + b + cu + dv \right)
\end{equation}
and  $a>0$,
then direct calculations show that for any $\theta \in \mathbb{R}$
$$
	\lim_{R \to \infty} \Phi \circ X \left( - \frac{c}{2a}+R\cos\theta, - \frac{d}{2a}+R\cos\theta \right)
	= \begin{pmatrix} a^{-1} & 0 \\ 0 & 0 \end{pmatrix}
	\in \qthreep\setminus\qthreept,
$$
which shows that the image of a sphere, a horosphere, or a hypersphere in the form of $X$ in \eqref{Eq:202603060429PM} with $a>0$ approach
as $t \mapsto \infty$ not a point in the ideal boundary but a point in the lightlike line $\qthreep\setminus\qthreept$.

Direct calculations show that the mean curvature of  the graph  given by \eqref{202506180754PM} and \eqref{Eq:202506190732AM} is
\begin{equation}\label{Eq:202506190735AM}
	H
	= \frac{-4ab+c^2+d^2}{2}
	= \frac{ \langle \vec{n}, \vec{n} \rangle }{2}.
\end{equation}
Thus $\Pi$ is timelike, lightlike or spacelike if and only if $H<0$, $H=0$ or $H>0$, respectively.
We immediately obtain the followings:
\begin{lemma}
	Let $X$ be a totally umbilic surface. Then,
	\begin{itemize}
		\item  $X$ has CMC $\h=0$ iff it is a horosphere.
		\item  $X$ has CMC $\h<0$ iff it is a sphere.
		\item  $X$ has CMC $\h>0$ iff it is a hypersphere.
	\end{itemize}
\end{lemma}

\begin{corollary}\label{lemma:converse}
	Every entire and totally umbilic graph must  have CMC $\h \le 0$.
\end{corollary}

\section{Bernstein-type theorems in $\mathbb{Q}^3_+$}

For the demonstration of ideas in this section, it is better to use $\omega := -\ln \tau$ instead of $\tau$.
In terms of $z:=u+iv$ and $\omega := -\ln \tau$, \eqref{Eq:202601180922AM} and \eqref{Eq:202601180923AM} may be written as
\begin{gather}
	\label{Eq:202601140915PM}
	(\mathbf{g}_{ij})  = \begin{pmatrix} e^{2\omega} & 0 \\ 0 & e^{2\omega} \end{pmatrix}, \qquad
	(\mathbf{A}_{ij})  = \begin{pmatrix} \omega_{uu} & \omega_{uv} \\ \omega_{uv} & \omega_{vv} \end{pmatrix}
	- \frac{1}{2} \begin{pmatrix} \omega_{u}^2-\omega_v^2 & 2\omega_u \omega_v \\ 2\omega_u \omega_v & -\omega_{u}^2+\omega_v^2 \end{pmatrix},
	\\
	\label{Eq:202507101519}
	Q = -\omega^2_z+\omega_{zz}\ ,\quad
	H = 2 e^{-2\omega} \omega_{z\bar{z}}\ ,\quad
	H^2-K=4 e^{-4\omega}|\omega^2_z-\omega_{zz}|^2\ge0.
\end{gather}

\subsection{Bernstein-type theorem for ZMC graphs in $\qthreep$}\label{Sec:3.2}

First, we consider entire ZMC graphs. The following assertion is a key to prove a Bernstein-type theorem for these surfaces.

\begin{lemma}\label{lemma:Schwarzian_ZMC}
	For any ZMC surface $X$ in $\qthreep$, there exists a holomorphic function $g$ with non-vanishing derivative such that the Hopf differential $Q$ and the function $\omega$ in \eqref{Eq:202601140915PM} satisfy
	\begin{equation}\label{eq:S(g)}
		S(g) =2Q\quad \text{and} \quad e^\omega = |g_z|,
	\end{equation}
	where $S(g)=\frac{g_{zzz}}{g_z}-\frac{3}{2}\left(\frac{g_{zz}}{g_z}\right)^2$ is the Schwarzian derivative of $g$.
	Moreover, the function $g$ is unique up to M\"obius transformations $g \mapsto \hat{g}=ag+b$ for constants $a,b \in \mathbb{C}$ with $|a|=1$.
\end{lemma}

\begin{proof}
	By the equation \eqref{Eq:202507101519}, $\omega$ is harmonic. Then there exists a holomorphic function $F$ such that $F_z=2\omega_z$. A desired function $g$ is given by $g=\int^z e^{F(w)}dw$.

	The uniqueness of $g$ follows from \eqref{eq:S(g)}. In fact, if a holomorphic function $\hat{g}$ satisfies \eqref{eq:S(g)}, then $|\hat{g}_z/g_z|=1$ holds by the second equation of \eqref{eq:S(g)}. The holomorphicity of $\hat{g}_z/g_z$ implies that $\hat{g}_z=ag_z$ for some constant $a$ with $|a|=1$. Therefore, we obtain $\hat{g}=ag+b$ for some $b\in \mathbb{C}$. For such transformations, the first equation of \eqref{eq:S(g)} is preserved since $S(g)$ is invariant under M\"obius transformations.
\end{proof}

\begin{proposition}\label{prop:zmc}
	Let $X$ be an entire ZMC graph in $\qthreep$. If the Gaussian curvature $K$ of $X$ is bounded below, then $X$ must be a horosphere.
\end{proposition}
\begin{proof}
	Let $g$ be given by Lemma \ref{lemma:Schwarzian_ZMC}. Then
	the equations \eqref{Eq:202507101519} and \eqref{eq:S(g)} imply the relation
	\begin{equation}\label{eq:K_S(g)}
		K = -\left| \frac{S(g)}{g_z^2}\right|^2.
	\end{equation}
	$S(g)/g_z^2$ is defined all over $\mathbb{C}$ and holomorphic everywhere.
	Since $K$ is bounded below, we may conclude that $S(g)/g_z^2$ is constant by the Liouville theorem.

	Let $M(z) := 1/g_z$. Then a straightfoward calculation shows that $S(g)/g_z^2$ can be written as
	$$
		-M M_{zz} + \frac{1}{2} M_z^2 = \frac{S(g)}{g_z^2}.
	$$
	Since it is a constant, by taking the derivative of it, we have
	\(
	-M_zM_{zz}-MM_{zzz}+M_zM_{zz}=0,
	\)
	that is, $MM_{zzz}=0$. Therefore, $M$ is
	\(
	M=c_0+c_1z+c_2z^2
	\)
	for some $c_i\in \mathbb{C}$ ($i=0,1,2$). By definition, $M(u,v) \not= 0$ at any $(u,v)$, and this forces
	\(
	c_1=c_2=0.
	\)
	In this case,  $M$, hence $g_z$ and $\omega$, are all constant by \eqref{eq:S(g)}. So $X$ is a horosphere.
\end{proof}

\begin{corollary}
	The Gaussian curvature of any entire ZMC graph in $\qthreep$ except horospheres takes all negative values. In particular,
	if an entire graph in $\qthreep$ has ZMC and its Gaussian curvature has a negative exceptional value, then it must be a horosphere.
\end{corollary}

\begin{proof}
	Let $g$ be given by Lemma \ref{lemma:Schwarzian_ZMC}.
	Then it satisfies \eqref{eq:K_S(g)}.
	Suppose that the holomorphic function $S(g)/g_z^2$ is not constant. Then, because of Picard's little theorem, the image of $S(g)/g_z^2$ must be either the entire complex plane or the complex plane minus a point. Therefore $K$ takes all negative values, which contradicts the assumption that $K$ has a negative exceptional value. So $K$ is constant. Finally, by Proposition \ref{prop:zmc}, we obtain the desired result.
\end{proof}
Similar results hold regarding the value distribution of the Gaussian curvature of complete spacelike CMC surfaces in the isotropic $3$-space $\mathbb{I}^3$ \cite{ALY1}.

\subsection{Bernstein-type Theorem for nonzero CMC graphs in $\mathbb{Q}^3_+$}
The graph of $\omega = \omega(z, \bar{z})$ in $\qthreep$ has CMC $\h$ if and only if
\begin{align}\label{Eq:202507101527}
	2 e^{-2\omega} \omega_{z\bar{z}}= \h .
\end{align}
See \eqref{Eq:202507101519}. If $\h \not= 0$ this is in fact the famous Liouville's equation and its solutions are all known. (\textit{cf}. \cite{BHL}, \cite{Liouville} for example.)

\begin{fact}[Solution of Liouville's equation]\label{Liouville}
	Let $\omega$ be defined on a simply connected domain $\mathcal{U}\subset\mathbb{C}$.
	Then,
	$ce^{-2\omega}\omega_{z\bar{z}}=-1$  for some constant  $c\in\mathbb{R}^+$
	if and only if
	there exist a meromorphic function
	$g(z):\mathcal{U}\to\hat{\mathbb{C}}:=\mathbb{C}  \cup \{\infty\}$ with $g_z\ne0$ for all $z$, such that
	\begin{align}\label{Eq:202507202002}
		e^{\omega}= \frac{\sqrt{c}|g_z|}{1+g\bar{g}}
	\end{align}
	which is unique up to M\"obius transformations $ g \mapsto  A \star g$ for $A \in SU(2)$.

	On the other hand,
	$ce^{-2\omega}\omega_{z\bar{z}}=1$ for some constant $c\in\mathbb{R}^+$
	if and only if there exist a meromorphic function $g(z):\mathcal{U}\to\hat{\mathbb{C}}\setminus\mathbb{S}^1$ with $g_z\ne0$ for all $z$, such that
	\begin{align}\label{Eq:202507202003}
		e^{\omega}= \frac{\sqrt{c}|g_z|}{1-g\bar{g}} \quad \text{for}\quad |g|<1, \quad \text{or}\quad e^{\omega}= \frac{\sqrt{c}|g_z|}{g\bar{g}-1}\quad \text{for}\quad |g|>1
	\end{align}
	which is unique up to M\"obius transformations $g \mapsto B \star g$ for $B \in SU(1,1)$.
\end{fact}
Here, the M\"obius transformation is
$$
	\begin{pmatrix} a & b \\ c & d \end{pmatrix} \star z := \frac{a z + b}{ c z + d}.
$$

If $\omega$ satisfies \eqref{Eq:202507202002} or \eqref{Eq:202507202003}, then the graph $X$ of $\omega$ has negative or positive CMC, respectively.

\begin{lemma}\label{coro:mean}
	For any nonzero real number $\h$, a CMC $\h$ graph $X:\mathcal{U}\to\qthreep$ over a simply connected open set $\mathcal{U}$
	can be represented as follows for some meromorphic function $g$ with $g_z (z) \ne0$ for all $z$:
	\begin{itemize}
		\item[(1)] For  $\h<0$,
		      \begin{equation}\label{Eq:202602200111AM}
			      X (z) =\frac{\sqrt{-2/\h}|g_z (z) |}{1+|g (z) |^2}\begin{pmatrix} z\bar{z} & z \\ \bar{z} & 1 \end{pmatrix}.
		      \end{equation}
		\item[(2)] For $\h>0$,
		      \begin{equation}\label{Eq:202602200112AM}
			      X (z) =\frac{\sqrt{2/\h}|g_z (z) |}{|1-|g (z) |^2|}\begin{pmatrix} z\bar{z} & z \\ \bar{z} & 1 \end{pmatrix}.
		      \end{equation}
		      In this case, $|g (z) |\ne1$ for any $z$.
	\end{itemize}
\end{lemma}
\begin{proof}
	A proof follows from the equation \eqref{Eq:202507101519} and  Fact~\ref{Liouville}.
	We take $c=-2/\h$ if $\h<0$ and $c=2/\h$ if $\h>0$.
\end{proof}
\begin{remark}\label{Rem:202601180320PM}
	For any constant $\h <0$, let
	$ t = \sqrt{-2\h} \cosh{u} $
	where $z=u+iv$, which is obtained from Lemma~\ref{coro:mean} \textit{(1)} with $g=e^z$. The graph of $t$ is entire and has CMC $\h$ but
	$K= \h^2( 1 - \cosh^4 u)$, which is not bounded below. See \eqref{Eq:202601180923AM} for necessary formulae.
\end{remark}

Now we establish a Bernstein-type theorem for nonzero CMC graphs in $\mathbb{Q}^3_+$. Because of the orientation of the lightlike Gauss map $G$  of $X$ in \eqref{Eq:202507142137}, we need to separately consider the cases where $\h$ is positive or negative.

\begin{lemma}\label{Lem:202602190936PM}
	For CMC surfaces given by \eqref{Eq:202602200111AM} or \eqref{Eq:202602200112AM}, the relation $S(g) = 2Q$ holds.
\end{lemma}
\begin{proof}
	A proof follows from direct calculations using \eqref{Eq:202601180923AM}.%
\end{proof}
\begin{proposition}\label{prop:negative}
	Let $\h$ be an arbitrary negative real number.
	Suppose that the Gaussian curvature $K$ of an entire CMC $\h$ graph in $\qthreep$ is bounded below.
	Then the image of $X$ is congruent to a part of the image of
	\begin{equation}\label{Eq:202507072027}
		\mathbb{C} \ni z \mapsto \frac{\sqrt{-2/\h}}{1+  z\bar{z}} \begin{pmatrix} z \bar{z} & z \\ \bar{z} & 1 \end{pmatrix}.
	\end{equation}
	It is in fact a hypersphere of $\qthreep$.
\end{proposition}
\begin{proof}
	If the entire graph $X$ in $\qthreep$ has negative CMC, then there exist a constant $c$ and a meromorphic function
	$g:\mathbb{C}\to\hat{\mathbb{C}}$ with $g_z(z) \ne 0$ for all $z$ such that
	\begin{align}\label{Eq:202511151536}
		e^{2\omega} = \frac{c g_z \overline{g_z}}{(1+ g\bar{g})^2}.
	\end{align}
	Then, using Lemma~\ref{Lem:202602190936PM}, we obtain that
	$$
		C_0 > H^2-K = 4 e^{-4\omega} |Q|^2 = (1+ g\bar{g})^4 \left| \frac{S(g)}{cg_z^2} \right|^2 > \left| \frac{S(g)}{cg_z^2} \right|^2.
	$$
	As before, $S(g)/g_z^2$ is defined all over $\mathbb{C}$ and holomorphic everywhere. So Liouville theorem implies that there exists a constant $C_1\in\mathbb{C}$ such that ${S(g)}/{(cg_z^2)}=C_1$ .
	Next, we show that $C_1=0$. When $g$ has a pole, the relation ${S(g)}/{(cg_z^2)}=C_1$ and holomorphicity of $S(g)$  imply $C_1=0$. When $g$ has no pole, we also obtain $C_1=0$ by the same argument of the ZMC case (see the proof of Proposition \ref{prop:zmc}).

	Recall that $S(g(z))=0$ is only and only if $g(z)= A \star z$ for some $A \in \sltc$,
	which with Corollary~\ref{coro:mean} implies that
	\[
		X(z,\bar{z}) = e^{\omega(z,\bar{z})} \begin{pmatrix} z\bar{z} & z \\ \bar{z} & 1 \end{pmatrix} = \sqrt{-\frac{2}{\h}}\frac{1 }{|az+b|^2+|cz+d|^2}\begin{pmatrix} z\bar{z} & z \\ \bar{z} & 1 \end{pmatrix}
	\]
	for some $a,b,c,d \in \mathbb{C}$ with $ad-bc=1$.
	Direct calculations show that
	$$
		A X (z, \bar{z}) A^* = \sqrt{-\frac{2}{\h}}\frac{1}{ w\bar{w} +1} \begin{pmatrix} w \bar{w} & w \\ \bar{w} & 1 \end{pmatrix}, \quad \text{where}\quad A := \begin{pmatrix} a & b \\ c & d \end{pmatrix}, \ w := A \star z.
	$$
	So we get the conclusion.
\end{proof}

Now we consider  entire positive CMC graphs in $\qthreep$.

\begin{proposition}\label{prop:positive}
	There exists no entire graph in $\qthreep$ which has positive CMC.
\end{proposition}
\begin{proof}
	Suppose  that there exists an entire function $\omega$ whose  graph in $\qthreep$  has  positive CMC $\h$. Then, from  Lemma~\ref{coro:mean}, there exists a meromorphic function $g:\mathbb{C}\to\hat{\mathbb{C}}\setminus\mathbb{S}^1$ with $g_z(z) \ne 0$ for all $z$  such  that
	\begin{align}\label{Eq:202507081352}
		e^{2\omega} = \frac{2 g_z \overline{g_z}}{ \h (1-g\bar{g})^2}.
	\end{align}
	Then, by continuity of $g$, either $|g(z)|<1$ for all $z\in\mathbb{C}$ or $|g(z)|>1$ for all $z\in\mathbb{C}$. In the former case, $g$ must be an  entire bounded function, hence is constant by Liouville theorem. In the latter case, $h(z):=1/g(z)$ is an entire bounded function, hence is constant. In both cases, $g$ is a constant function, but then $g_z(z) = 0$ for all $z \in \mathbb{C}$, which is a contradiction.
\end{proof}

\begin{proof}[A proof of the Main Theorem]
	A proof follows from Propositions \ref{prop:zmc}, \ref{prop:negative}, \ref{prop:positive} and Section~\ref{Spheres}.
\end{proof}
\begin{proof}[A proof of Corollary~\ref{Coro:202603071137AM}]
	The condition implies that $S$ is compact, hecne the Gaussian curvature is bounded below. Then $S$ minus a point satisfies the condition of the Main theorem. Hence, it is either a sphere or a horosphere. But a horosphere does not project to the entire ideal boundary. It is clear that a sphere satisfies the condition of the statement.
\end{proof}

\begin{remark}
	Recall the relation $S(g)=2Q$ between the holomorphic or meromorphic function $g$ and the Hopf differential $Q$ in Lemmas \ref{lemma:Schwarzian_ZMC} and \ref{Lem:202602190936PM}.
  
        $\bullet$ For CMC $1$ surfaces in the hyperbolic $3$-space $\mathbb{H}^3$, a similar relation $S(g)-S(G)=2Q$ is known, where $G$ is the hyperbolic Gauss map and $g$ is the secondary Gauss map of a CMC $1$ surface in $\mathbb{H}^3$. See \cite{UY1}.

        $\bullet$ In the forthcoming paper \cite{ALY3}, the authors show that there is a correspondence between CMC $H$ surfaces with $H<,=$ or $>0$ and (spacelike) ZMC surfaces in $\mathbb{E}^3$,  $\mathbb{I}^3$,  $\mathbb{L}^3$, respectively, of which $g$ becomes the Gauss map.
\end{remark}


\end{document}